\documentclass[11pt,a4paper]{amsart}
\usepackage[latin1]{inputenc}

\usepackage{amssymb,amsmath,amscd,amsfonts,amsthm,mathrsfs}
\usepackage{times} 
\usepackage[all]{xy}

\newcommand{\Cc}{\mathcal{C}}

\newcommand{\Mcc}{\mathcal{M}}

\newcommand{\Hr}{\mathscr{H}}

\newcommand{\ess}{{\rm{ess}}}

\newcommand{\N}{\mathbb{N}}
\newcommand{\R}{\mathbb{R}}
\newcommand{\C}{\mathbb{C}}

\newcommand{\D}{\mathbb{D}}
\newcommand{\T}{\mathbb{T}}
\newcommand{\ds}{\displaystyle}
\newcommand{\si}{\sigma}
\newcommand{\al}{\alpha}
\newcommand{\la}{\lambda}
\newcommand{\ap}{\approx}

\newcommand{\Hc}{\mathcal{H}}
\newcommand{\Lc}{\mathcal{L}}
\newcommand{\Sg}{\textbf{S}}
\newcommand{\Id}{{\rm{Id}}}
\newcommand{\im}{{\rm{Im} }}

\newcommand{\bsl}{\backslash}

\newtheorem{theorem}{Theorem}[section]
\newtheorem{proposition}[theorem]{Proposition}
\newtheorem{lemma}[theorem]{Lemma}

\numberwithin{equation}{section}

\begin{document}

\title[Eigenvalues of a Dirac operator]{On quantitative bounds on eigenvalues of a complex perturbation of a Dirac operator}

\author{Cl\'{e}ment Dubuisson}
\address{Institut de Mathematiques de Bordeaux
Universite Bordeaux 1
351, cours de la Lib\'eration
F-33405 Talence cedex}
\email{clement.dubuisson@math.u-bordeaux1.fr, dubuissonc@gmail.com}

\begin{abstract}
We prove a Lieb-Thirring type inequality for
a complex perturbation of a $d$-dimensional massive Dirac
operator $D_m, m\geq 0, d \geq 1$ whose spectrum is
$]-\infty,-m]\cup[m,+\infty[$. The difficulty of the study is that the
unperturbed operator is not bounded from below in this case, and, to
overcome it, we use the methods of complex function 
theory. The methods of the article also give similar results for complex perturbations of  the Klein-Gordon operator.
\end{abstract}

\subjclass{30C35, 35Q41, 47A10, 47B10, 47B25, 81Q10}
\keywords{Dirac operator, complex perturbation, discrete spectrum, Lieb-Thirring type inequality, conformal mapping, perturbation determinant}

\maketitle


\section{Introduction}

In the Dirac formalism (e.g., \cite[section 1]{Th}) the
properties of a relativistic particles with spin-$1/2$ (for instance
electrons in the massive case and neutrinos in the non-massive case)
are described with the help of the Dirac operator. Because of spin
structure, the configuration space of the particle takes values in
$\C^n$, where $n=2^\nu$ with $\nu \geq 1$. The movement of the free
particle of mass $m$ is given by the Dirac equation,
\begin{equation*}
\mathrm{i} \hbar \frac{\partial \varphi}{\partial t }= D_m \varphi, 
\end{equation*}
where $\varphi\in L^2(\R^{d}; \C^{n})$
with $d\in \{1,\ldots,n-1\}$, if $m>0$ and 
$d\in \{1,\ldots,n\}$ otherwise.
The Dirac operator is defined as 
\begin{equation}\label{def:Dm}
D_m:=-\mathrm{i} c \hbar\, \alpha\cdot \nabla +mc^2\beta=
-\mathrm{i} c\hbar\sum_{k=1}^d\alpha_k \dfrac{\partial }{\partial x_k} + mc^2\beta. 
\end{equation}
Here $c$ is the speed of light, and $\hbar$ is the reduced Planck
constant. We renormalize and consider $\hbar=c=1$. 
Here we set ${\alpha}:=\left(\alpha_1, \ldots, \alpha_d \right)$
and $\beta:=\alpha_{d+1}$. The matrices $\alpha_i$ are $d+1$ linearly
independent self-adjoint linear maps, acting in $\C^{n}$, satisfying
the following anti-commutation relations
\begin{equation*}
\alpha_i\alpha_j+\alpha_j\alpha_i=2\delta_{i,j}\mathrm{Id}\,,
\end{equation*}
where $i,j=1,\dots, d+1.$
For instance, on $\R^3$, one can choose the Pauli-Dirac representation 
\begin{gather*}
\alpha_i=\left( \begin{array}{cc}
0 & \si_i \\
\si_i & 0
\end{array} \right),
\quad
\beta=\left(\begin{array}{cc}
{\rm{Id}}_{\C^2} & 0 \\
0 & -{\rm{Id}}_{\C^2}
\end{array} \right),
\end{gather*}
where $i=1,2, 3,$ and
\begin{gather*}
\sigma_1=\left(
\begin{array}{cc}
0 & 1 \\
1 & 0
\end{array} \right),\quad
\quad\sigma_2=\left(
\begin{array}{cc}
0 & -\mathrm{i} \\
\mathrm{i} & 0
\end{array} \right),\quad
\quad\sigma_3=\left(
\begin{array}{cc}
1 & 0 \\
0 & -1
\end{array} \right).
\end{gather*}
In the general case, the $n\times n$-matrices $\al_j$ are constructed
as special elements of the so-called Clifford algebra (see 
\cite[Chapter 1]{Ob}).
Without any loss of generality we take
\[\beta:=\left(\begin{array}{cc}
{\rm{Id}_{\C^{n/2}}} & 0\\
0 & -{\rm{Id}_{\C^{n/2}}}
\end{array}\right).\]
Mimicking the proofs of section 1.1 to section 1.4 of \cite[section 1]{Th}
it is easy to check that 
the operator $D_m$ is essentially self-adjoint on $\Cc^\infty_c(\R^d;
\C^{n})$ and the domain of its closure is $\Hr^1(\R^d; \C^{n})$,
the Sobolev space of order $1$ with values in $\C^{n}$. The closure of
the operator is denoted with the same symbol $D_m$. With the help of
the Fourier transform, it is easy to prove that $D_m$ is
unitarily equivalent to 
\begin{equation}\label{e:equi}
\left(\begin{array}{cc}
\sqrt{-\Delta_{\R^d}+m^2} \times {\rm{Id}_{\C^{n/2}}} & 0\\
0 & -\sqrt{-\Delta_{\R^d}+m^2}\times {\rm{Id}_{\C^{n/2}}}
\end{array}\right).
\end{equation}
Therefore the spectrum of $D_m$ is purely absolutely continuous and is
given by ${]-\infty, -m]\cup[m,+\infty[}$. 

Another object of interest for us is the following operator
given by
\begin{equation}\label{def:KG}
K_m= \sqrt{-\Delta_{\R^d}+m^2}\times {\rm{Id}_{\C^{l}}}
\end{equation}
with $m\geq 0$ and $l\geq 1$.
We call it \emph{Klein-Gordon operator} but there are other possible conventions for its name.
This time, the index $l$ is not related to $d$.
It is well known that it describes a massive relativistic particle without spin;
naturally enough, this is just ``a half" of the Dirac operator 
{in the view of \eqref{e:equi}}.
One can readily see that, as above, it is essentially self-adjoint on $\Cc^\infty_c(\R^d; \C^{l})$, the domain of its closure is $\Hr^1(\R^d; \C^{l})$. The closure of the operator being denoted by the same symbol, its spectrum is absolutely continuous and equals $[m,+\infty[$.

The purpose of this article is to obtain a Lieb-Thirring type
inequality for the discrete spectrum of a complex perturbation of
\eqref{def:Dm} and \eqref{def:KG}.
We actually concentrate on the Dirac operator,
and the case of Klein-Gordon operator will follow easily from the obtained results.
We would like to mention that the problems of this kind
(for perturbations of various self-adjoint operators)
were rather intensively studied over the last years.
We refer to papers by Frank, Laptev, Lieb, and Seiringer (\cite{FrLaLiSe}),
Bruneau and Ouhabaz (\cite{BrOu}),
Borichev, Golinskii, and Kupin (\cite{BoGoKu}),
Demuth, Hansmann, and Katriel (\cite{DeHaKa, DeHaKa1, DeHaKa2}),
Golinskii and Kupin (\cite{GoKu1, GoKu2}),
Hansmann (\cite{Ha1, Ha2}),
and Hansmann and Katriel (\cite{HaKa}).
An appropriate modification of some methods of the above papers was applied by Sambou (\cite{Sa})
to the study of a complex perturbation of a magnetic Schr\"odinger operator.
Moreover, an interesting recent paper by Cuenin, Laptev, and Tretter
(\cite{CuLaTr}) studies not only the distribution,
but also the localization of the discrete spectrum of a complex perturbation of
one-dimensional Dirac operator $D_m, \ m\geq 0$. 

The results of papers \cite{FrLaLiSe} and \cite{GoKu} were obtained by reducing the case of a complex perturbation of a given operator to a self-adjoint situation;
by the way, the paper \cite{FrLaLiSe} also contains the discussion of the properties of complex perturbations of self-adjoint operators and exhaustive list of references on it.
For several reasons detailed below, the approach of the present article is different.
It is rather close to \cite{BoGoKu} and \cite{DeHaKa} and it is based on complex function theory. 
First, as explained in \cite{FrLaLiSe} and \cite{GoKu},
the methods of the theory of self-adjoint operators lead to the results on the part of the discrete spectrum of the perturbed operator in certain angular domains,
and here we are interested in distributional characteristics of the whole discrete spectrum.
Second, if one wants to follow the approach of \cite{FrLaLiSe} (or \cite{GoKu}),
then one generically should consider a part of the discrete spectrum in a domain touching the lower (or the upper) edge of the essential spectrum of the operator
and then to apply a version of Birman-Schwinger principle (\cite{Bi, Sch})
together with Fan-Mirski lemma (\cite[chapter III]{Bha})
to get some information on it.
Since the operators we consider are not semi-bounded, the above techniques do not apply immediately.
Of course, one may think about an appropriate decoupling and a use of a generalized version of Birman-Schwinger principle in the spirit of nice recent work by Frank and Simon (\cite{FrSi}).
The authors of this paper treat the case of a Dirac operator on $\R$, which is not enough for our purposes.
Probably, a development of methods of \cite{FrSi} would allow one to advance on Lieb-Thirring type inequalities for multi-dimensional Dirac operators,
but we do not pursue this direction here.

To formulate our results, we introduce some notation.
For a (possibly unbounded) operator $A$ on a separable Hilbert space {$\Hc$},
we denote the spectrum, the essential and the discrete spectrum of $A$ by
$\si(A), \si_{\ess}(A),$ and $\si_d(A)$, respectively.
Here the discrete spectrum is the set of all eigenvalues which are discrete points of the spectrum whose corresponding eigenspaces (or rootspaces) are finite dimensional.
The essential spectrum is then the complement of the discrete spectrum in the spectrum of $A$. For more details, see \cite[subsection VII.3]{ReSi1} or \cite[p.5]{DeHaKa2}.
We put $\Sg_p,p\geq 1$ to be the \emph{Schatten-von Neumann} class of compact operators,
see section \ref{s21} for the definitions and discussion of the object.

Let $\Mcc_{n,n}(\C)$ denote the space of $n\times n$ complex-valued matrices.
For $p\geq 1$, consider the space of $\Mcc_{n,n}(\C)$-valued measurable functions on $\R^d$ defined as
\begin{align}
L^p(\R^d; \Mcc_{n,n}(\C))&=\left\{V  :  \|V\|_{L^p}^p=\int_{\R^d} \|V(x)\|^p_F\, dx\right\},
\end{align}
where $\|\cdot\|_F$ is the \emph{Frobenius norm},
\begin{equation}\label{def:frob}
\|V(x)\|_F=\left(\sum_{i,j=1,\ldots,  n}|(V(x))_{i,j}|^2\right)^{1/2}.
\end{equation}

The function $V$ is often identified with the operator of multiplication by itself.
Assuming that $V\in L^p(\R^d; \Mcc_{n,n}(\C))$ and $p>d$,
we prove (see Proposition \ref{det-RS}) that
the multiplication by $V$ is \emph{relatively Schatten-von Neumann perturbation} of $D_m$,
\textit{i.e.}, dom($D_m$) $\subset$ dom($V$), and
\begin{equation}\label{hyp}
 V(\la-D_m)^{-1} \in \Sg_p,
\end{equation}
for one $\la \in
\C\backslash\si(D_m)$ (and hence for all these $\la$'s).
Consider the perturbed operator
\begin{equation} \label{hyp1}
D=D_m+V. 
\end{equation}
Recall that by Weyl's theorem on essential spectrum (\cite[Theorem XIII.14]{ReSi4},
or equivalently \cite[Theorem IX.2.1]{EdEv})
\begin{equation*}
\si_{\ess}(D)=\si_{\ess}(D_m)=\si(D_m)=]-\infty,-m]\cup [m,+\infty[.
\end{equation*}

Our main results are the following theorems.

\begin{theorem}[case $m>0$]\label{t1}
Let $D$ be the Dirac operator defined in \eqref{hyp1} and ${V \in
L^p(\R^d; \Mcc_{n,n}(\C)), \ p>d}$. Then its discrete spectrum $\si_d(D)$ admits the
following Lieb-Thirring type bound: for all $0<\tau <\min \{p-d,
1\}$, 
\begin{equation}\label{eq01}
 \ds\sum_{\la \in \si_d(D)} \dfrac{d(\la,\si(D_m))^{p+\tau}}
{|\la-m|\cdot |\la+m|(1+|\la|)^{2p-2+2\tau}}
\leq C \|V\|^p_{L^p},
\end{equation}
except when $d=1$ and $1<p<2$ in which case the Lieb-Thirring type bound becomes
\begin{equation}\label{eq02}
 \ds\sum_{\la \in \si_d(D)} \dfrac{d(\la,\si(D_m))^{p+\tau}}
{(|\la-m|\cdot|\la+m|)^{(p+\tau)/2}(1+|\la|)^{p+\tau}}
\leq C \|V\|^p_{L^p},
\end{equation}
where the constants $C$ depends on $n, d, p, m$, and $\tau$.
\end{theorem}

It seems appropriate to compare this claim to results of Cancelier, L\'evy-Bruhl and Nourrigat (\cite{CaLeNo}) and Frank and Simon (\cite{FrSi}).
These papers are devoted to the case of a self-adjoint perturbation,
where relation \eqref{eq01} can be rewritten in a simpler form.
Indeed, in this case the discrete spectrum $\si_d(D)$ lies in $]-m, m[$.
Defining $E_m=\{\pm m\}$, we see that $d(\la, \si(D_m))=d(\la, E_m)$ for $\la\in \si_d(D)$,
and an easy computation shows that \eqref{eq01} reads as
\begin{equation}\label{eq011}
\ds\sum_{\la \in \si_d(D)} d(\la, E_m)^{p-1+\tau}
\leq C \|V\|^p_{L^p},
\end{equation}
and \eqref{eq02} as
\begin{equation}\label{eq021}
\ds\sum_{\la \in \si_d(D)} d(\la, E_m)^{(p+\tau)/2}
\leq C \|V\|^p_{L^p}.
\end{equation}
In \cite{CaLeNo}, $d=3$ and, with our notation, one of the central results of the paper
(see \cite[corollary 1.3]{CaLeNo}) says
\begin{equation}\label{i:CaLeNo}
\ds\sum_{\la \in \si_d(D)} d(\la, E_m)^{p-3}
\leq C_p  \big( \|V\|^p_{L^p}+ \|V\|^{p-3/2}_{L^{p-3/2}}\big),
\end{equation}
where $p>3$ and $V\in L^p(\R^3; \R)\cap L^{p-3/2}(\R^3; \R)$.

In \cite{FrSi}, $d=1$ and one has (see \cite[Theorem 7.1]{FrSi})
\begin{equation}\label{i:FrSi}
\ds\sum_{\la \in \si_d(D)} d(\la, E_m)^{p-1}
\leq C_{1,p}  \|V\|^p_{L^p} + C_{2,p,m} \|V\|^{p-1/2}_{L^{p-1/2}},
\end{equation}
where $p\geq 3/2$ and $V\in L^{p}(\R; \R)\cap L^{p-1/2}(\R; \R)$.
Clearly enough, \eqref{i:CaLeNo} and \eqref{i:FrSi} are stronger than \eqref{eq011} (and \eqref{eq021});
the gap between \eqref{i:FrSi} and \eqref{eq011} seems to be smaller than the one between \eqref{i:CaLeNo} and \eqref{eq011}.
On the other hand, even for real-valued case, bound \eqref{eq011} is valid for larger classes of potentials.
The point is that its proper rewriting \eqref{eq01} remains true even for complex-valued perturbations.
As often happens, the strength of the method we use is indivisible from its weakness, \textit{i.e.}, being very general and rather powerful,
it does not go ultimately far in exploiting the specifics of operators under consideration.
Hence the bounds on the discrete spectrum it produces are expected to be improvable at least in some special cases.

The version of Theorem \ref{t1} for $m=0$ is as follows.
\begin{theorem}[case $m=0$]\label{t2}
Let $D$ be the Dirac operator defined in \eqref{hyp1} with $m=0$ and $V\in L^p(\R^d; \Mcc_{n,n}(\C)), \ p>d$.
Then 
\begin{equation} \label{eq03}
\ds\sum_{\la \in \si_d(D)}
\dfrac{d(\la,\si(D_0))^{p+\tau}}{(1+|\la|)^{2(p+\tau)}}
\leq C\|V\|_{L^p}^p,
\end{equation}
where $0<\tau<\min \{p-d, 1\}$ and $C$ depends on $n, d, p, m$, and $\tau$.
\end{theorem}

Now, consider the perturbed Klein-Gordon operator
\[
K=K_m+V. \]
Using the computations done for the perturbed Dirac operator we obtain the following results.
\begin{theorem}[case $m>0$]\label{t3}
Let $K$ be the Klein-Gordon operator defined above and
$V\in L^p(\R^d; \Mcc_{l,l}(\C))$, $p>d$.
Then, for $0<\tau$ small enough, we have
\begin{equation}\label{eq030}
 \ds\sum_{\la \in \si_d(K)} \dfrac{d(\la,\si(K_m))^{p+\tau}}
{|\la-m|\, (1+|\la|)^{p+\max\{p/2,d\}+2\tau-1}}
\leq C \|V\|^p_{L^p},
\end{equation}
where the constant $C$ depends on $l, d, p$, and $\tau$.
\end{theorem}

We observe that, for $m=0$, the operator under consideration is
\[K=(-\Delta)^{1/2}+V.\]
We observe that a non-trivial degeneration of a bound on the resolvent of $K_0$ takes place in this case and the inequality of Theorem \ref{t3} can be refined in the following way:

\begin{theorem}[case $m=0$]\label{t4}
Let $K$ be the Klein-Gordon operator defined above with $m=0$ and
$V\in L^p(\R^d; \Mcc_{l,l}(\C))$, $p>d$. Then, for $0< \tau$ small enough, we have
\begin{equation}\label{eq040}
 \ds\sum_{\la \in \si_d(K)} \dfrac{d(\la,\si(K_0))^{p+\tau}}
{|\la|^{\min\{(p+\tau)/2,d\}}(1+|\la|)^{\frac{p}{2}+\max\{p,2d\}-d+2\tau}}
\leq C \|V\|^p_{L^p},
\end{equation}
where the constant $C$ depends on $l, d, p$, and $\tau$.
\end{theorem}

For the self-adjoint case, an account on Lieb-Thirring inequalities
for perturbations of the so-called \emph{fractional Schr\"odinger operators} $(-\Delta)^s$
with power $0<s<\min\{1,d/2\}$, can be found in Frank, Lieb and Seiringer (\cite{FrLiSe})
and Lieb and Seiringer (\cite[chapter 4]{LiSe}).
It is convenient to compare Theorem \ref{t4} with these results for $s=1/2$.
Of course, $\si_d(K)$ lies on the negative real half-axis in this case.
In our notation Theorem 2.1 from \cite{FrLiSe} says
\begin{equation}\label{i:FrLiSe}
\sum_{\la\in \si_d(K)} |\la|^{p-d}\le C_{p,d} \|V_-\|^p_{L^p},
\end{equation}
where $p>d$ and $V_-=\min\{V, 0\}$.
Since $d(\la,\si(K_0))=|\la|$, bound \eqref{eq040} looks like
\begin{equation}\label{eq041}
 \ds\sum_{\la \in \si_d(K)} \dfrac{|\la|^{\max\{(p+\tau)/2,\ p+\tau-d\}}}
{(1+|\la|)^{p/2+\max\{p,2d\}-d+2\tau}}
\leq C \|V\|^p_{L^p},
\end{equation}
which is slightly weaker than \eqref{i:FrLiSe} as we will see immediately.
Indeed, if we have $p+\tau-d>(p+\tau)/2$ (or, equivalently, $(p+\tau)/2>d$), the left hand-side of \eqref{eq041} is
\begin{equation*}
\dfrac{|\la|^{p+\tau-d}}{(1+|\la|)^{p/2+\max\{p,2d\}-d+2\tau}}
\leq |\la|^{p-d}\cdot \left(\frac{|\la|}{1+|\la|}\right)^\tau \cdot \frac 1{(1+|\la|)^{\dots}}
\leq |\la|^{p-d}.
\end{equation*}
If $(p+\tau)/2<d$, we have for the left hand-side of \eqref{eq041}
\begin{align*}
\dfrac{|\la|^{(p+\tau)/2}}{(1+|\la|)^{p/2+d+2\tau}}
& \leq \frac{|\la|^{p+\tau-d}\cdot |\la|^{-(p+\tau)/2+d}}{(1+|\la|)^{p/2+d+2\tau}}\\
& =
|\la|^{p-d}\cdot \left(\frac{|\la|}{1+|\la|}\right)^\tau\cdot \frac{ |\la|^{-(p+\tau)/2+d}}{(1+|\la|)^{p/2+d+\tau}},
\end{align*}
and a simple bound using $d-(p+\tau)/2>0$ yields that the second and the third factor in the above formula are less or equal to one. It is clear that the remarks preceding Theorem \ref{t2} transposes verbatim to this situation.

The case of the Klein-Gordon operator also can be treated  with the help of a result 
(\cite[Cor. 1]{Ha2}).  In fact we find the following :
\begin{theorem}\label{t5}
Let $K$ be the Klein-Gordon operator defined above with $m\geq 0$ and $V \in
L^p(\R^d; \Mcc_{n,n}(\C))$, $p>d$.
Then its discrete spectrum $\si_d(K)$ admits the bound: 
\begin{equation}\label{KGbyH}
 \ds\sum_{\la \in \si_d(K)} \dfrac{d(\la,\si(K_m))^{p}}
{(1+|\la|)^{2p}}
\leq C \|V\|^p_{L^p},
\end{equation}
where the constant $C$ depends on $n, d, p$ and $m$.
\end{theorem}

To prove Theorem \ref{t5} we apply Corollary 1 from \cite{Ha2} to the resolvents
$(\mu-K)^{-1}$ and $(\mu-K_m)^{-1}$ with $\mu$ appropriately chosen,
and use conformal maps and bounds of $\Sg_p$-norm to get the inequality.

One can see that relations \eqref{eq030} and \eqref{eq040} are in some sense stronger than \eqref{KGbyH} for $p>d$.

Before going to the discussion of obtained results, we say a couple
more words on the notation. Constants will be generic, \textit{i.e.}, changing
from one relation to another. Usually, they will be denoted by $C$ or
``const". For two strictly positive functions $f, g$ defined on a
domain $\Omega$ of the complex plane $\C$, we write $f(\la) \ap
g(\la)$ if the functions are comparable in the sense of the two-sided
inequality, \textit{i.e.}, there are constants $C_1,C_2>0$ so that $C_1
f(\la)\le g(\la)\le C_2 f(\la)$ for all $\la\in \Omega$. The choice of
the domain $\Omega$ will be clear from the context.

Theorems \ref{t1} and \ref{t2} provide quantitative estimates for the convergence of sequences of eigenvalues $(\la_n)\subset \si_d(D)$ to $\si_{\ess}(D)$ {for $V \in L^p(\R^d)$}. To illustrate, we fix $m>0$ and consider sequences 
$(\la_n)\subset \si_d(D)$ converging to a point $\la$ chosen in three different ways. Suppose that $\mathrm{Im} \la_n> 0$.
\begin{enumerate}
\item Let $\la=\pm m$ and assume there is a constant $C$ strictly positive such that
$|{\mathrm{Re}}(\la_n \mp m)|\le C\, |\mathrm{Im} \la_n|$. Then 
\begin{align*}
d(\la_n, \si(D_m)) \ap |\la_n \mp m|,
\quad |\la_n \pm m| \ap \mbox{ const},
\quad 1+|\la_n| \ap \mbox{ const},
\end{align*}
and relation \eqref{eq01} implies that
\[\ds\sum_{n=1}^{\infty}  |\la_n - m|^{p-1+\tau} < +\infty. \]

\item Let $\la=\infty$ and $|\mathrm{Im}(\la_n)|\leq C$.
Then
\begin{align*}
d(\la_n, \si(D_m)) \ap |{\mathrm{Im}}(\la_n)|,
\; |\la_n +m|.|\la_n -m| \ap |\la_n|^2,
\; 1+|\la_n| \ap |\la_n|,
\end{align*}
and relation \eqref{eq01} implies that
\[\ds\sum_{n=1}^{\infty} \dfrac{|{\mathrm{Im}}(\la_n)|^{p+\tau}}{|\la_n|^{2p+2\tau}} < +\infty. \]

\item If $\la\in ]m; \infty[$, then
\begin{align*}
d(\la_n, \si(D_m)) \ap |{\mathrm{Im}}(\la_n)|,
\; |\la_n +m|.|\la_n - m| \ap \mbox{ const},
\; 1+|\la_n| \ap \mbox{ const},
\end{align*}
and relation \eqref{eq01} implies that
\[\ds\sum_{n=1}^{\infty} |{\rm{Im}}(\la_n)|^{p+\tau} < +\infty. \]
\end{enumerate}

We conclude the introduction with few words on the structure of the paper.
The preliminary results are presented in section 2.
Section 3 contains the discussion of certain conformal maps appearing in the proofs.
Section 4 deals with a special perturbation determinant and corresponding bounds.
Theorems \ref{t1} and \ref{t2} are proved in sections \ref{s:proof} and \ref{massnull}, respectively. Since the proofs of Theorem \ref{t3} and \ref{t4} go exactly along the lines of Theorems \ref{t1} and \ref{t2}, it is omitted.

\medskip

\noindent\textbf{Acknowledgments}: I am grateful to Stanislas Kupin, 
Sylvain Gol\'enia and Vincent Bruneau for turning my attention to the problem and useful
discussions. I thank the anonymous referee for careful reading the manuscript and helpful remarks.
This research is partially supported by Franco-Ukrainian programm ``Dnipro 2013-14".


\section{Preliminaries}


\subsection{Schatten classes and determinants}\label{s21}
The contents of this subsection closely follow the monographs by Gohberg-Krein \cite{GoKr} and Simon \cite{Si1}.

For a separable Hilbert space $\Hc$, let $\Lc(\Hc)$ denote the space of bounded linear operators on $\Hc$.
We denote the class of compact operators on $\Hc$ by $\Sg_{\infty}$.
The Schatten-von Neumann classes $\Sg_p, 1\leq p < \infty,$ of compact operators are defined by
\[\Sg_p:=\{A \in \Sg_{\infty}, \|A\|_{\Sg_p}^p:=\ds\sum_{n=1}^{+\infty}s_n(A)^p<+\infty \},\]
where $s_n(A)$ is the $n$-th singular value of $A$.

For $A\in \Sg_n, n \in \N^*$, one can define the regularized determinant
\[ {\det}_n(\Id-A):=\ds\prod_{k=1}^{+\infty} \left[
(1-\la_k)\exp\left(\ds\sum_{j=1}^{n-1}\dfrac{\la_k^j}{j} \right) \right], 
\]
where $(\la_k)_k$ is the sequence of eigenvalues of $A$. 
This determinant has the following well-known properties (see \cite[Chap. IV]{GoKr} or \cite{Si1}):

\begin{enumerate}
\item $\det_n(\Id)=1$.

\item $\Id-A$ is invertible if and only if $\det_n(\Id-A)\neq 0$.

\item For any $A,B \in \Lc(\Hc)$ with $AB, BA \in \Sg_n, \det_n(\Id-AB)=\det_n(\Id-BA)$.

\item If $A(\cdot)$ is a holomorphic operator-valued function on a domain $\Omega$,
then the function $\det_n(\Id-A(\cdot))$ is also holomorphic on $\Omega$.

\item Let $A \in \Sg_p$ for some real $p\geq 1$.
Obviously, $A \in \Sg_{\lceil p \rceil}$,
where $\lceil p \rceil$ is defined by $\min\{n \in \N, n\geq p \}$, and the following inequality holds
\[ \label{i:det} |{\det}_{\lceil p \rceil}(\Id-A)|\leq
\exp\left(\Gamma_p\|A\|_{\Sg_p}^p\right),\]
where $\Gamma_p$ is a positive constant \cite[Theorem 9.2]{Si2}.
\end{enumerate}

For $A, B \in \Lc(\Hc)$ with $B-A \in \Sg_p$, we define the $\lceil p \rceil$-regularized perturbation determinant of $B$ with respect to $A$ by 
\[ d(\la):={\det}_{\lceil p \rceil}\left((\la-A)^{-1}(\la-B)\right)
={\det}_{\lceil p \rceil}(\Id-(\la-A)^{-1}(B-A)).
\]
This is a well defined holomorphic function on $\rho(A):=\C\backslash\si(A)$.

Furthermore, $\la \in \rho(A)$ is an eigenvalue of $B$ of multiplicity $k$ if and only if $\la$ is a zero of $\la\mapsto d(\la)$ of the same multiplicity.


\subsection{Theorem of Borichev-Golinskii-Kupin}
The following theorem, proved in \cite[Theorem 0.2]{BoGoKu}, gives a bound on the zeros of a holomorphic function on the unit disc $\D=\{|z|<1\}$ in terms of its growth towards the boundary $\T:=\{|z|=1\}$.
An important feature of this theorem is that it enables to take into account the existence of 'special' points $(\zeta_j)$ on the boundary of the unit disc,
where the function grows faster than at generic points.

\begin{theorem}\label{BGK}
Let $h$ be a holomorphic function on $\D$ with $h(0)=1$.
Assume that $h$ satisfies a bound of the form
\[ |h(z)|\leq\exp\left(
\dfrac{K}{(1-|z|)^{\al}} \ds\prod_{j=1}^N\dfrac{1}{|z-\zeta_j|^{\beta_j}}
\right),
\]
where $|\zeta_j|=1$ and $\al, \beta_j\geq 0, \ j=1,\dots, N$.

Then for any $\tau>0$ the zeros of $h$ satisfy the inequality
\[ \ds\sum_{h(z)=0}(1-|z|)^{\al+1+\tau}
\ds\prod_{j=1}^N|z-\zeta_j|^{(\beta_j-1+\tau)_+}\leq C\cdot K,
\]
where $C$ depends on $\al,\beta_j,\zeta_j$ and $\tau$.
\end{theorem}
Above, $x_+=\max\{x,0\}$. An other useful version of the above result is given in Hansmann-Katriel (\cite[Theorem 4]{HaKa}).


\section{Conformal mappings}\label{s3}

Throughout this section $m$ is positive.
Recall that $\T=\{z: |z|=1\}$ is the unit circle and $\D=\{z: |z|<1\}$ the open unit disc.
The idea is to send the resolvent set of
$D_m,\, \rho(D_m)=\C\backslash\{]-\infty,-m]\cup[m,+\infty[\}$
on the unit disc $\D$ via a conformal map and to obtain a comparison between the distance to the spectrum of $D_m$ and the one to the unit circle: this kind of comparison is called distortion.
We note by $d(z,A):=\ds\inf_{w\in A}|z-w|$ the distance between $z$ and $A$.

The map we are interested in is constructed as a composition of four ``elementary" conformal maps which are as follows:

\begin{enumerate}
\item $z_1=\dfrac{\la-m}{\la+m}:\C\backslash \si(D_m) \to \C\backslash [0,+\infty[$.
The inverse mapping is given by $\la=m\dfrac{1+z_1}{1-z_1}$.

\item $z_2=\sqrt{z_1}:\C\backslash [0,+\infty[\to \{{\rm{Im}}(z)>0\}$.
The inverse mapping is $z_1=z_2^2$.

\item $z_3=\dfrac{z_2-\mathrm{i}}{z_2+\mathrm{i}}:\{{\rm{Im}}(z)>0\} \to \D$.
The inverse map is $z_2=\mathrm{i} \dfrac{1+z_3}{1-z_3}$.

\item The normalization is operated by
\[
u=e^{i\theta}\dfrac{z_3-z_b}{1-\overline{z_b}z_3}:\D_{z_3}\to \D_u,
\]
where $z_b={-\mathrm{i} b}/(|m+\mathrm{i} b|+m)$ is the image of $\mathrm{i} b$ by the three first conformal mappings. As above, we sometimes label the unit disk $\D$ by the corresponding variable to avoid misunderstanding. We put furthermore 
\[
u_{m,+}:=u(1), \quad u_{m,-}:=u(-1).
\]
The inverse map is $z_3=\dfrac{u+e^{\mathrm{i}\theta}z_b}{e^{\mathrm{i}\theta}+u\overline{z_b}}$.
\end{enumerate}

Notice that the conformal mapping $u$ will serve to match the normalization $h(0)=1$ from Theorem \ref{BGK}.
The following conformal maps
\begin{align}\label{def:phi}
\psi & = (z_3\circ z_2\circ z_1)^{-1}:\D_{z_3} \rightarrow \C\backslash\si(D_m),\\
\varphi & = (u\circ z_3\circ z_2\circ z_1)^{-1}:\D_u \rightarrow \C\backslash\si(D_m) \nonumber
\end{align}
will be important for the sequel.

The map $\psi$ is easy to compute,
\begin{equation}\label{psi}
\la=\psi(z_3)=-2m\, \dfrac{z_3}{1+z_3^2}.
\end{equation}
The following technical propositions are essentially application of Koebe distortion theorem \cite[Corollary 1.4]{Po} to the map $\psi$.

\begin{lemma}\label{i:u} With the above notation, we have
\begin{enumerate}
\item $d(u,\mathbb{T})\ap d(z_3,\mathbb{T})$.
\item $|z_3-a|\ap |u-u(a)|$, where $a \in \{ 1,-1,\mathrm{i},-\mathrm{i} \}$.
\end{enumerate}
\end{lemma}

The proof of the lemma is obvious and hence is omitted.

\begin{proposition}[From $\C\bsl \si(D_m)$ to $\D$]\label{CM:1}
We have
\[ d(\la,\si(D_m)) \ap \dfrac{|u-u_{m,+}|\cdot|u-u_{m,-}|}
{|u-u(\mathrm{i})|^2|u-u(-\mathrm{i})|^2}\, d(u,\mathbb{T}).\]
\end{proposition}

\begin{proof}
Since $\psi'(z)=-2m\dfrac{1-z^2}{(1+z^2)^2}$,
we obtain by Koebe distortion theorem
\begin{align}\label{K}
\dfrac{m}{2}\dfrac{|1-z_3^2|}{|1+z_3^2|^2}(1+|z_3|)d(z_3,\T)
&\leq  d(\la,\si(D_m))\\ 
&\leq  2m\dfrac{|1-z_3^2|}{|1+z_3^2|^2}(1+|z_3|)d(z_3,\T).\nonumber
\end{align}
That is,
\[ d(\la,\si(D_m))\ap \dfrac{|1-z_3^2|}{|1+z_3^2|^2}(1+|z_3|)d(z_3,\T).
\]
Now,
\[
|1-z_3^2|=|1-z_3|\cdot|1+z_3|,\,
|1+z_3^2|=|z_3-\mathrm{i}|\cdot|z_3+\mathrm{i}|,\,
1\leq 1+|z_3|\leq 2,
\]
and we use the previous lemma to conclude the proof.
\end{proof}

\begin{proposition}[From $\D$ to $\C\bsl\si(D_m)$]\label{CM:2}
The following relation holds true
\[ d(u,\T)\ap \dfrac{d(\la,\si(D_m))}
{(|\la+m|\cdot|\la-m|)^{\frac{1}{2}}(1+|\la|)}.
\]
\end{proposition}

\begin{proof}
From \eqref{K}, we have
\[ \dfrac{d(\la,\si(D_m))|1+z_3^2|^2}{2m|1-z_3^2|(1+|z_3|)}
\leq 1-|z_3| \leq
\dfrac{2d(\la,\si(D_m))|1+z_3^2|^2}{m|1-z_3^2|(1+|z_3|)}, \]
and
\[ d(z_3,\T) \ap
d(\la,\si(D_m))\dfrac{1}{1+|z_3|}\dfrac{|1+z_3^2|^2}{|1-z_3^2|} \]
since $1/2 \leq (1+|z_3|)^{-1}\leq 1 $. 

The definitions of the maps $z_i, i=1, 2, 3$ easily imply that
\[
1-z_3^2=\dfrac{4\mathrm{i}\sqrt{z_1}}{(\sqrt{z_1}+\mathrm{i})^2},\;
1+z_3^2=\dfrac{2z_1-2}{(\sqrt{z_1}+\mathrm{i})^2},\;
|\sqrt{z_1}+\mathrm{i}|^2 \ap 1+|z_1|,
\]
where ${\rm{Im}}(\sqrt{z_1})={\rm{Im}}(z_2)>0$. Furthermore,
\[|z_1-1|= \dfrac{2m}{|\la+m|},\;
|\sqrt{z_1}|=\left|\dfrac{\la-m}{\la+m}\right|^{\tfrac{1}{2}},\;
\dfrac{1}{1+|z_1|}=\dfrac{|\la+m|}{|\la+m|+|\la-m|}.
\]

Putting all this together, we obtain
\begin{align*}
d(z_3,\T) & \ap
d(\la,\si(D_m))\dfrac{|z_1-1|^2}{|\sqrt{z_1}|(1+|z_1|)}\\
& \ap \dfrac{d(\la,\si(D_m))}
{(|\la+m|\cdot|\la-m|)^{\tfrac{1}{2}}(|\la+m|+|\la-m|)}\\
& \ap \dfrac{d(\la,\si(D_m))}
{(|\la+m|\cdot|\la-m|)^{\tfrac{1}{2}}(1+|\la|)},
\end{align*} 
and Lemma \ref{i:u} finishes the proof.
\end{proof}


\section{Perturbation determinant}


\subsection{A special perturbation determinant}
This subsection closely follows \cite[Section 3.1.1]{DeHaKa};
the holomorphic function $f:\C\bsl\si(D_m)\rightarrow\C$
is defined by a relation similar to the formula preceding \cite[formula (22)]{DeHaKa}.
For the sake of completeness, we give a short list of analytic properties of this function $f$ relating it to the properties of the operator $D$;
more details on these connections (and proofs) are in the quoted section of \cite{DeHaKa}.

Let $b$ be large enough to guarantee that $(-\mathrm{i}b+D)$ is invertible (see Lemma \ref{lem}).
We require that $V\in L^p(\R^d,\Mcc_{n,n}(\C)), p>d$, and,
as we will see in section \ref{s42}, this condition implies that
$V(\la-D_m)^{-1}\in \Sg_p$ for $\la \in \rho(D_m)$.
We consider the operator
\begin{equation}\label{def:F}
F(\la):=(\la-\mathrm{i} b)(-\mathrm{i} b+D)^{-1}V(\la-D_m)^{-1},
\end{equation}
and the holomorphic function
\begin{equation}\label{def:f}
f(\la):={\det}_{\lceil p \rceil}(\Id-F(\la)).
\end{equation}

It is not difficult to see that:
\begin{enumerate}
\item 
The operator-valued function $F$ is well-defined and $F(\la)\in \Sg_p, \ p\ge 1$. Consequently, $f$ is well-defined and holomorphic on $\rho(D_m)$ as well.
\item 
Recording an alternative representation
\[
{\rm{Id}}-F(\la)=\left[{\rm{Id}}-(\la-\mathrm{i} b)(-\mathrm{i} b+D)^{-1} \right]
\left[{\rm{Id}}-(\la-\mathrm{i} b)(\mathrm{i} b+D_m)^{-1}\right]^{-1},
\]
we deduce that $\mathrm{Id} -F(\la)$ is not invertible if and only if $\la\in \si_d(D)$. Moreover, the multiplicity of the zero $\la_0$ of $f$ exactly coincides with the algebraic multiplicity of the eigenvalue $\la_0$ of the operator $D$, $\la_0\in \si_d(D)$.
\item
The above relation also yields that $F(\mathrm{i} b)=0$, and $f(\mathrm{i} b)=1$.
\end{enumerate}


\subsection{Schatten bounds on the operator $V(\la-D_m)^{-1}$}\label{s42} 
The choice of the Frobenius norm in definition \eqref{def:frob} is important for the next proposition.

\begin{proposition}\label{det-RS}
Let $V\in L^p(\R^d; \Mcc_{n,n} (\C)), \ p>d$, and $\la\in\rho(D_m)$.
Define $\mu_m:\R^d\to \C^n$ by
$\mu_m(x):=\sqrt{|x|^2+m^2}\times \mathrm{Id}_{\C^n}$.
Then 
$V(\la-D_m)^{-1}\in \Sg_p$, and
\[
\|V(\la-D_m)^{-1}\|_{\Sg_p}^p  \leq
(2\pi)^{-d}\|V\|_{L^p}^p\cdot\|(\la-\mu_m(\cdot))^{-1}\|_{L^p}^p
\]
if ${\rm{Re}}(\la)\geq 0$ and
\[
 \|V(\la-D_m)^{-1}\|_{\Sg_p}^p  \leq (2\pi)^{-d}
\|V\|_{L^p}^p\cdot\|(\la+\mu_m(\cdot))^{-1}\|_{L^p}^p
\]
if ${\rm{Re}}(\la)\leq 0$. 
\end{proposition}

\begin{proof}
For $x\in\R^d$, we set, to be brief,
\[
f(x)=V(x),\quad g(x)=(\la-\mu_m(x))^{-1}\times \mathrm{Id}_{\C^n}.
\]
We prove actually that
\begin{equation}\label{RS1}
\|f \cdot g(-\mathrm i \nabla)\|_{\Sg_p}^p  \leq
(2\pi)^{-d}\|f\|_{L^p}^p\cdot\|g\|_{L^p}^p
\end{equation}
for $1\leq p<\infty$.
The proof closely follows \cite[Theorem 4.1]{Si2},
the main modifications being the use of the Frobenius norm
and relation \eqref{RS2} for matrix-valued integral operators.
To stress the differences of the matrix-valued case as compared to the scalar one,
we give the argument in a somewhat more detailed form than the quoted theorem from \cite{Si2}.

All operators considered in this proposition act on $L^2(\R^d; \C^n)$.
 Let $B$ be a bounded operator given by
\[
(Bf)(x)=\int_{\R^d} K(x,y) f(y)\, dy,
\]
where the kernel $K(\cdot,\cdot)$ is a $\Mcc_{n,n}(\C)$-valued measurable function
and $f$ belongs to $L^2(\R^d; \C^n)$.
A familiar result from \cite{GoKr} or \cite{Si2} says that
\begin{equation}\label{RS2}
\|B\|^2_{\Sg_2}=\int_{\R^d}\int_{\R^d} \|K(x,y)\|^2_F \, dxdy.
\end{equation}
Now, denote by
$A:=f(x)g(-\mathrm{i} \nabla)$ the integral operator
associated to the kernel \[(2\pi)^{-d/2}f(x)\check{g}(x-y),\] 
where $\check{g}$ is the inverse Fourier transform of $g$.

Suppose that $f$ and $g$ are in $L^2(\R^d; \Mcc_{n,n} (\C))$.
Recalling \eqref{def:frob} and the fact that the norm is submultiplicative, we obtain that 
\begin{align*}
\|A\|^2_{\Sg_2}&=\|f(x)g(-\mathrm{i} \nabla)\|_{\Sg_2}^2=
(2\pi)^{d}\int_{\R^d}\int_{\R^d}
\|f(x)\check{g}(x-y)\|_F^2\, dxdy 
\\
&\leq (2\pi)^{-d}\int_{\R^d}\int_{\R^d}
\|f(x)\|^2_F \|\check{g}(x-y)\|_F^2\, dxdy 
\\
&\leq
(2\pi)^{-d}\|f\|_{L^2}^2\cdot\|g\|_{L^2}^2,
\end{align*}
where we used Fubini and Fourier-Parseval theorems.
So, the integral operator $A$ lies in $\Sg_2$ (\textit{i.e.}, it is Hilbert-Schmidt), and we have bound \eqref{RS1} for $p=2$.
In particular, $A$ is a compact operator.

Recall that $L^\infty$ is the space endowed with the norm
\[\|f\|_{L^\infty}:= \textrm{ess-sup}_{x\in \R^d} \|f(x)\|_{F}.\]
Let us take two test functions $\phi, \psi$ from $L^2(\R^d; \C^n)$ such that 
\[
\|\phi\|^2_{L^2}=\int_{\R^d}\|\phi(x)\|^2_2\, dx=\int_{\R^d}\left(\sum^n_{i=1} |\phi_i(x)|^2\right)\, dx\leq 1
\]
and $\|\psi\|^2_{L^2}\leq 1$. We are to prove that
\begin{equation}\label{RS3}
\|A\|=\|A\|_{\Sg_\infty}=\sup_{\phi,\psi}|(\phi, A\psi)|\leq \|f\|_\infty\cdot \|g\|_\infty,
\end{equation}
where $f,g\in L^2\cap L^\infty$. Indeed,
\begin{align*}
|(\phi, A\psi)|&=|(\phi, f (g \hat\psi)\check{})|=|(f^*\phi, (g \hat\psi)\check{}\, )|
\\
&\leq \|f^*\phi\|_{L^2}\,  \|(g \hat\psi)\check{}\, \|_{L^2}.
\end{align*}
Then
\[
\|f^*\phi\|^2_{L^2}=\int_{\R^d} \|f^*\phi\|^2_2\, dx\leq \int_{\R^d} \|f\|^2_F \|\phi\|^2_2\, dx\leq \|f\|^2_\infty \|\phi\|^2_{L^2},
\]
and, similarly,
\[
\| g\hat\psi\|^2_{L^2}\leq \|g\|^2_\infty \|\hat \psi\|^2_{L^2}=\|g\|^2_\infty \|\psi\|^2_{L^2}.
\]
Hence, \eqref{RS3} is proved for all $f, g$ in $L^2\cap L^\infty$. 
Then the standard complex interpolation argument yields
\[\|f(x)g(-\mathrm{i} \nabla)\|_{\Sg_p}^p\leq (2\pi)^{-d}\|f\|_{L^p}^p\cdot\|g\|_{L^p}^p,\]
for all $2\leq p < \infty$. The same result for indices $1\leq p\le 2$ follows by duality.
\end{proof}


\subsection{Bound on the resolvent}\label{s43}
In this subsection, we bound expressions
$\|(\la \pm\mu_m(\cdot))^{-1}\|_{L^p}$ appearing in Proposition \ref{det-RS}. 
\begin{proposition}\label{det-BR}
Let $\la=\la_0+ \mathrm{i} \la_1$ and $p>d$.
Then
\[
 \|(\la-\mu_m(\cdot))^{-1}\|_{L^p}^p \leq
 \dfrac{K_1}{d(\la,\si(D_m))^{p-1}}(1+|\la-m|^{d-1})
\]
 for $\la_0\geq 0,$ and 
\[
 \|(\la+\mu_m(\cdot))^{-1}\|_{L^p}^p \leq
 \dfrac{K_2}{d(\la,\si(D_m))^{p-1}}(1+|\la+m|^{d-1}),
\]
for $\la_0\leq 0.$ Above, the constants $K_1$ and $K_2$ depend on $n, d, p, m$.
\end{proposition}

\begin{proof}
First of all, recall that
$(\la-\mu_m(\cdot))^{-1}=(\la-\mu_m(\cdot))^{-1}\times\mathrm{Id}_{\C^n}$, hence
$\|(\la-\mu_m(\cdot))^{-1}\|_{L^p(\R^d,\Mcc_{n,n}(\C))}^p=
n^{p/2} \|(\la-\mu_m(\cdot))^{-1}\|_{L^p(\R^d,\C)}^p$.

The cases $\pm \la_0\geq 0$ being similar, we give the proof for $\la_0 \geq 0$ only.  After a change of variable, we are reduced to bound
\[
I=\ds\int_0^{+\infty} \dfrac{r^{d-1}}{|\sqrt{r^2+m^2}-\la|^p}\,dr.
\]
We write
$|\sqrt{r^2+m^2}-\la|^p=
\left((\sqrt{r^2+m^2}-\la_0)^2+\la_1^2\right)^{\tfrac{p}{2}}$
and make the change of variable $s=\sqrt{r^2+m^2}-m$.
Hence,
\begin{equation} I=\ds\int_0^{+\infty}
\dfrac{((s+m)^2-m^2)^{\tfrac{d-2}{2}}(s+m)}
{((s+m-\la_0)^2+\la_1^2)^{\tfrac{p}{2}}}\,ds. \label{e:int}
\end{equation}

We now distinguish the cases $m\leq \la_0$ \ and \ $0\leq \la_0 < m$.
For $m\leq \la_0$, observe that $d(\la,\si(D_m))=|\la_1|$. We let $\beta=\la_0-m \geq 0$ and use the inequality
$\sqrt{(s+m)^2-m^2}\leq s+m$, so
\begin{align*}I \leq
\ds\int_0^{+\infty}\dfrac{(s+m)^{d-1}}{((s-\beta)^2+\la_1^2)^{\tfrac{p}{2}}}\,ds.
\end{align*}
\noindent Since $m\leq \la_0$ and $\la \notin \si(D_m)$,
we have $|\la_1|>0$, and 
\begin{align}
\ds\int_0^{+\infty}\dfrac{(s+m)^{d-1}}{((s-\beta)^2+\la_1^2)^{\tfrac{p}{2}}}
& =  \dfrac{1}{|\la_1|^p} \ds\int_0^{\beta}
\dfrac{(s+m)^{d-1}}
{\left(\left(\dfrac{s-\beta}{\la_1}\right)^2+1\right)^{\tfrac{p}{2}}}\,ds \nonumber \\
& + 
\dfrac{1}{|\la_1|^p} \ds\int_{\beta}^{+\infty}
\dfrac{(s+m)^{d-1}}
{\left(\left(\dfrac{s-\beta}{\la_1}\right)^2+1\right)^{\tfrac{p}{2}}}\,\label{sum:int}  ds.
\end{align}
In the right hand-side of \eqref{sum:int}, we make the change of variable $t=\dfrac{\beta-s}{\la_1}$ in the first integral
and $t=\dfrac{s-\beta}{\la_1}$ in the second one.
Then we apply the inequality ${(a+b)^{d-1}\leq C_d(a^{d-1}+b^{d-1})}$ for $a,b\geq 0$.
This leads to the bounds
\begin{align*}
I & \leq  \dfrac{C_d}{|\la_1|^{p-1}}
\left(\ds\int_0^{\tfrac{\beta}{\la_1}}
\dfrac{(\beta-\la_1 t)^{d-1}dt}{(t^2+1)^{\tfrac{p}{2}}}
+\ds\int_0^{\tfrac{\beta}{\la_1}}
\dfrac{m^{d-1}dt}{(t^2+1)^{\tfrac{p}{2}}}\right.
\\
& \hspace*{15mm}
\left. + \ds\int_{0}^{+\infty}
\dfrac{(\beta+\la_1 t)^{d-1}dt}{(t^2+1)^{\tfrac{p}{2}}}
+\ds\int_{0}^{+\infty}
\dfrac{m^{d-1}dt}{(t^2+1)^{\tfrac{p}{2}}}
\right).
\end{align*}
Recalling $p>d$, we continue as 
\[
I\leq
\dfrac{C_d}{|\la_1|^{p-1}} \left(
2(\beta^{d-1}+m^{d-1})\ds\int_0^{+\infty}\dfrac{dt}{(t^2+1)^{\tfrac{p}{2}}}
+2|\la_1|^{d-1} \ds\int_0^{+\infty}\dfrac{t^{d-1}\,dt}{(t^2+1)^{\tfrac{p}{2}}}
\right).
\]
Using $(\beta^{d-1}+|\la_1|^{d-1}) \ap |\la-m|^{d-1}$, we get to
\begin{equation}
 I \leq \dfrac{1}{d(\la,\si(D_m))^{p-1}}
\left(K_1 |\la-m|^{d-1}+K_2\right)
\label{e:int1}
\end{equation}
for $m\leq \la_0$.

We now turn to the case $0 \leq \la_0<m$. We see $d(\la,\si(D_m))=|\la-m|$; going back to $\eqref{e:int}$, we use the inequality $(s+m-\la_0)^2+\la_1^2\geq s^2+|\la-m|^2$.
Hence 
\[
I \leq \dfrac{1}{|\la-m|^p}
\ds\int_0^{+\infty}\dfrac{(s+m)^{d-1}}{\left(\left(\dfrac{s}{|\la-m|}\right)^2+1 \right)^{\tfrac{p}{2}}}\,ds.
\]
Doing the change of variable $t=\dfrac{s}{|\la-m|}$ and bounding as in the first part of the computation, we come to
\begin{equation}
I \leq \dfrac{1}{d(\la,\si(D_m))^{p-1}}
\left(\tilde{K}_1 |\la-m|^{d-1}+\tilde{K}_2\right)
\label{e:int2}
\end{equation}
for $0\leq \la_0<m$.

The proposition is proved.
\end{proof}


\section{Proof of the main result}\label{s:proof}

Let us start with the following lemma.
\begin{lemma}\label{lem}
For $p>d$ and $b$ large enough, we have $\|(-\mathrm{i} b+D)^{-1}\|\leq1$.
\end{lemma} 

\begin{proof}
First, notice that the inequality $\|V(\mathrm{i} b-D_m)^{-1}\| < 1$ yields that
the operator $(-\mathrm{i} b+D)$ is invertible.

Indeed, the inequality $\|V(\mathrm{i} b-D_m)^{-1}\| < 1$ implies that
$\mathrm{Id}-V(\mathrm{i} b-D_m)^{-1}$ is invertible,
and we have
\begin{align*}
\mathrm{Id}-V(\mathrm{i} b-D_m)^{-1} & =
(\mathrm{i} b-D_m)(\mathrm{i} b-D_m)^{-1}-V(\mathrm{i} b-D_m)^{-1}\\
& =(\mathrm{i} b-D_m-V)(\mathrm{i} b-D_m)^{-1}\\
&= -(-\mathrm{i} b+D)(\mathrm{i} b-D_m)^{-1}.
\end{align*}

Second, we show that we have
$\|V(\mathrm{i} b-D_m)^{-1}\| < 1$ for $b$ large enough.
Since $\|A\| \leq \|A\|_{\Sg_p}$ for any operator $A$, 
Propositions \ref{det-RS} and \ref{det-BR} entail
\begin{align}
\|V(\mathrm{i} b-D_m)^{-1}\| &\leq \|V(\mathrm{i} b-D_m)^{-1}\|_{\Sg_p} \nonumber\\
& \leq K\|V\|_{L^p}\dfrac{(1+|\mathrm{i} b-m|^{d-1})}{|\mathrm{i} b-m|^{p-1}}, \label{eq06}
\end{align}
where the constant $K$ does not depend on $b$.
It is convenient to put
\[
C(b)=K\|V\|_{L^p}\dfrac{(1+|\mathrm{i} b-m|^{d-1})}{|\mathrm{i} b-m|^{p-1}}.
\]
The right-hand side of inequality \eqref{eq06} trivially goes to zero when $b$ goes to infinity, and so
$\|V(\mathrm{i} b-D_m)^{-1}\| \leq C(b) < 1$ for $b$ large enough.

Now we prove that $\|(-\mathrm{i} b+D)^{-1}\|\leq 1$ for $b$ large enough.
Using the resolvent identity, we get 
\begin{align*}
\|(-\mathrm{i} b+D)^{-1}\| & \leq \|(-\mathrm{i} b+D_m)^{-1}\|
+\|(-\mathrm{i} b+D)^{-1}\|\cdot\|V(-\mathrm{i} b+D_m)^{-1}\|\\
 & \leq \|(-\mathrm{i} b+D_m)^{-1}\|
+\|(-\mathrm{i} b+D)^{-1}\|\cdot\|V(-\mathrm{i} b+D_m)^{-1}\|_{\Sg_p}.
\end{align*}
Since $D_m^*=D_m$,
\[\|(-\mathrm{i} b+D_m)^{-1}\|=\dfrac{1}{d(\mathrm{i} b, \si(D_m))}
=\dfrac{1}{|\mathrm{i} b-m|},\]
and as above we obtain
\[
\|(-\mathrm{i} b+D)^{-1}\| \leq \dfrac{1}{|\mathrm{i} b-m|}
+C(b)\, \|(-\mathrm{i} b+D)^{-1}\|.
\]
Resolving this inequality with respect to $\|(-\mathrm{i} b+D)^{-1}\|$,
we get the claim of the lemma.
\end{proof}

\noindent\textit{Proof of Theorem \ref{t1}}:
Recall from \eqref{def:f} that
$f(\la)={\det}_{\lceil p \rceil}(\Id-F(\la))$,
with
\[F(\la):=(\la-\mathrm{i} b)(-\mathrm{i} b+D)^{-1}V(\la-D_m)^{-1} \in \Sg_p.\]
We have by the property of the regularized determinant
\begin{align} |f(\la)|\leq \exp\left(\Gamma_p \|(\la-\mathrm{i} b)(D-\mathrm{i} b)^{-1}V(\la-D_m)^{-1}\|_{\Sg_p}^p\right).\label{i:f}
\end{align}
Applying Propositions \ref{det-RS} and \ref{det-BR} to \eqref{i:f}, we get to 
\begin{align*}
\log |f(\la)| \leq
K\|V\|_{L^p}^p\|(-\mathrm{i} b+D)^{-1}\|^p
\dfrac{|\la-\mathrm{i} b|^p(1+|\la-m|^{d-1})}{d(\la,\si(D_m))^{p-1}}.
\end{align*}
for ${\rm{Re}}(\la)\geq 0$.
Up to obvious changes, a similar expression is obtained for the case ${\mathrm{Re}}(\la)\leq 0$.

We continue as 
\[
|\la-\mathrm{i} b|\le C(1+|\la|), \quad (1+|\la-m|)\le C(1+|\la|),
\]
and the factor $\|(-\mathrm{i} b+D)^{-1}\|^p$ is bounded from above with the help of Lemma \ref{lem}.
So,
\begin{equation}\label{i:proof}
\log |f(\la)| \leq K\|V\|_{L^p}^p\, 
\dfrac{(1+|\la|)^{p+d-1}}{d(\la,\si(D_m))^{p-1}}.
\end{equation}

We now have to go in $\D$ in order to apply Theorem \ref{BGK}.
That is, recalling defini\-tions \eqref{def:phi}, we consider the function $g(u)=f\circ \varphi(u)$;
it is trivially holomorphic on $\D_u$.
The considerations of section \ref{s3} and relation \eqref{psi} entail
\[
1+|\la| \ap \dfrac{|1-z_3|^2+|1+z_3|^2}{|z_3-\mathrm{i}|\cdot|z_3+\mathrm{i}|}
 \ap \dfrac{1+|z_3|^2}{|z_3-\mathrm{i}|\cdot|z_3+\mathrm{i}|}.
\]
In particular, we have by Lemma \ref{i:u}
\begin{equation*}
1+|\la| \ap \dfrac{1}{|u-u(\mathrm{i})|\cdot|u-u(-\mathrm{i})|}.
\end{equation*}
By the previous relation, \eqref{i:proof}, and Proposition \ref{CM:1}, we obtain
\begin{equation*}
\log |g(u)| \leq K\cdot \|V\|_{L^p}^p\
\dfrac{|u-u(i)|^{p-d-1}|u-u(-i)|^{p-d-1}}
{d(u,\T)^{p-1}|u-u_{m,+}|^{p-1}|u-u_{m,-}|^{p-1}}
\end{equation*}

By assumptions of the theorem, we always have $p>d$.
Consider first the case $p-d\geq 1$,
or, equivalently, $p-d-1\geq 0$.
Obviously, the factors
$|u-u(i)|^{p-d-1}$ and $|u-u(-i)|^{p-d-1}$ are then bounded,
and applying Theorem \ref{BGK}, we find for $0<\tau <1$
\begin{equation}\label{e07}
\ds\sum_{g(u)=0}(1-|u|)^{p+\tau}|u-u_{m,+}|^{p-2+\tau}|u-u_{m,-}|^{p-2+\tau}
\leq C\cdot \|V\|_{L^p}^p,
\end{equation}
where $C$ depends on $n, d, p, m$ and $\tau$.

In the second case, we have $0<p-d< 1$ or $-1<p-d-1< 0$.
We use Theorem \ref{BGK} with $0<\tau< p-d$ and so 
\begin{equation}\label{e08}
\ds\sum_{g(u)=0}(1-|u|)^{p+\tau}|u-u_{m,+}|^{(p-2+\tau)_+}|u-u_{m,-}|^{(p-2+\tau)_+}
\leq C_1\cdot \|V\|_{L^p}^p,
\end{equation}
where $C_1$ depends on $n, d, p, m$ and $\tau$.

The last step of the proof consists in transferring relations \eqref{e07}, \eqref{e08} back to $\rho(D_m)=\C\bsl\si(D_m)$.
Remind that we have by Lemma \ref{i:u} and Proposition \ref{CM:2}
\begin{align*}
1-|u|=d(u,\T) & \ap \dfrac{d(\la,\si(D_m))}{(|\la+m|\cdot|\la-m|)^{1/2}(1+|\la|)},\\
|u-u_{m,+}|\cdot|u-u_{m,-}| & \ap \dfrac{\left(|\la-m|\cdot|\la+m|\right)^{1/2}}{1+|\la|}.
\end{align*}

Thus, if $p\geq 2$, we come to
\begin{align*}
(1-|u|)^{p+\tau}\left(|u-u_{m,+}|\cdot|u-u_{m,-}|\right)^{p-2+\tau}
 & \geq \\
& \dfrac{C\,d(\la,\si(D_m))^{p+\tau}}{|\la+m|\cdot|\la-m|(1+|\la|)^{2(p+\tau-1)}},
\end{align*}
and if $1<p<2$, for $\tau >0 $ small enough, we have
\[
(1-|u|)^{p+\tau} \geq 
\dfrac{C\,d(\la,\si(D_m))^{p+\tau}}{(|\la+m|\cdot|\la-m|)^{(p+\tau)/2}(1+|\la|)^{p+\tau}}.
\]
The claim of the theorem follows. \hfill $\Box$ 

\medskip

Of course, one can wonder what happens if we choose $\tau\geq p-d$ in the case of relation \eqref{e08}.
It is easy to see that Theorem \ref{BGK} still applies, but, rather expectedly, the inequality obtained in this way is weaker than \eqref{e08},
so we do not pursue this direction.


\section{The case of $m=0$}\label{massnull}
The method is the same but the spectrum of $D_0$ is the whole $\R$, $\si(D_0)=\R$.
The slight differences as compared to the case $m>0$ come from the study of the conformal mappings
and the computation of the Schatten norm of the resolvent $V(\la-D_0)^{-1}, \la\in \rho(D_0)$.
Since the techniques and computations are extremely similar (not to say almost identical) to the case of Theorem \ref{t1},
we give only a fast sketch of Theorem \ref{t2}.

As the conformal map concerns, notice that $\rho(D_0)=\C^+\cup\C^-$,
where $\C^\pm$ are the sets $\{\la \in \C: \pm\im(\la)>0\}$.
So we can compute the contributions of the discrete spectrum $\si_d(D)\cap \C^\pm$ to \eqref{eq03} and then add them up.
That is why, without loss of generality,
we discuss the case of $\la \in \si_d(D)\cap \C^+$,
and the case of $\si_d(D)\cap \C^-$ is treated similarly.
The conformal map $\varphi$ we are interested in, is particularly simple
\begin{align*}
\la & = \varphi(u)=\mathrm{i} b\,  \frac{1+u}{1-u}: \D_u \to {\C^+}_\la,\\
u & = \varphi^{-1}(\la)= \frac{\la-\mathrm{i}b}{\la+\mathrm{i}b}: {\C^+}_\la \to \D_u.
\end{align*}
For instance, the distortions become
\[
d(\la,\si(D_0)) \ap \dfrac{d(u, \T)}{|u-1|^2}, \quad
d(u, \T)  \ap \dfrac{d(\la,\si(D_0))}{(1+|\la|)^2}.
\]

Let, as before, $p>d$. For $\la\in\C^+$, the bound on the resolvent reads as
\[
\|V (\la-D_0)^{-1}\|^p_{\Sg_p}\leq C \|V\|^p_{L^p}\, \|(\la-\mu_0(x))^{-1}\|^p_{L^p},
\]
where $\mu_0(x)=|x|$, and we need to bound the integral
\[I=\ds\int_0^{+\infty}\dfrac{r^{d-1}}{|r-\la|^p}\, dr.\]
Similarly to the computation of section \ref{s43}, we get
\begin{equation} \label{BR0}
I \leq \dfrac{K}{d(\la,\si(D_0))^{p-1}}\cdot |\la|^{d-1}
\end{equation}
and then
\begin{equation*}
\|V (\la-D_0)^{-1}\|^p_{\Sg_p}\leq C \|V\|^p_{L^p}\,
\dfrac{|\la|^{d-1}}{d(\la,\si(D_0))^{p-1}}.
\end{equation*}

\noindent \textit{Sketch of the proof of Theorem \ref{t2}.}
By property of the perturbation determinant in $\Sg_p$, we have
\[ \log|f(\la)| \leq K \|V\|_{L^p}^p 
\dfrac{|\la-{\mathrm{i}}b|^p|\la|^{d-1}}{d(\la;\si(D_0))^{p-1}},
\]
where $f$ is defined in \eqref{def:f} and $F$ is the same as in \eqref{def:F} with $m=0$.
Writing $\la=\varphi(u)$ and $g=f\circ \varphi$, we see
\begin{align*}
\log |g(u)| & \leq K \|V\|_{L^p}^p 
\dfrac{|u|^p|1+u|^{d-1}}{|1-u|^{d+p-1}}\cdot
\dfrac{|1-u|^{2(p-1)}}{d(u,\T)^{p-1}}\\
 & \leq K \|V\|_{L^p}^p \dfrac{|u|^p|1+u|^{d-1}}{|1-u|^{d-p+1}d(u,\T)^{p-1}}.
\end{align*}
We apply Theorem \ref{BGK} to the function $g$ to obtain
\[
\ds\sum_{g(u)=0}d(u,\T)^{p+\tau}|u-1|^{(d-p+\tau)_+} \leq K \|V\|_{L^p}^p
\]
for $\tau>0$ small enough.
Using the properties of the maps $\varphi, \varphi^{-1}$ discussed at the beginning of this subsection,
we conclude the proof of the theorem. \hfill $\Box$

\end{document}